\documentclass[3p,12pt]{elsarticle}
\usepackage[ruled,vlined]{algorithm2e}
\usepackage{hyperref}
\usepackage{umoline}
\usepackage{color}
\usepackage{array}
\usepackage{ctable}
\usepackage[all]{xy}

\newcommand{\reduline}[1]{}

\newcommand{\rcolored}[1]{}

\usepackage{amssymb}

 \usepackage{amsmath}
 \usepackage{amsthm}
\usepackage{multirow}
 \newtheorem{thm}{Theorem}[section]
\newtheorem{ex}[thm]{Example}
 \newtheorem{cor}[thm]{Corollary}
 \newtheorem{lem}[thm]{Lemma}
 \newtheorem{prop}[thm]{Proposition}
   
 \theoremstyle{definition}
 \newtheorem{defn}[thm]{Definition}
 \theoremstyle{remark}
 \newtheorem{rem}[thm]{Remark}

 \theoremstyle{definition}
 \newtheorem{exm}{Example}[section]
 \def\Blem{\begin{lem}}
 \def\Elem{\end{lem}}
 \newcommand{\uccomment}[1]{}
 \def\Bpr{\begin{prop}}
 \def\Epr{\end{prop}}

 \def\Bp{\begin{proof}}
 \def\Ep{\end{proof}}

 \def\Bex{\begin{exm}}
 \def\Eex{\end{exm}}

 \def\Bcor{\begin{cor}}
 \def\Ecor{\end{cor}}

 \def\Br{\begin{rem}}
 \def\Er{\end{rem}}
 \def\Bthm{\begin{thm}}
 \def\Ethm{\end{thm}}
 \def\Bd{\begin{defn}}
 \def\Ed{\end{defn}}
 \def\Beq{\begin{equation}}
 \def\Eeq{\end{equation}}

 \newcommand{\To}{\longrightarrow}

\newcommand{\tx}{\tilde{X}}

  \def\NN{\mathbb{N}}

\journal{ }

\begin{document}

\begin{frontmatter}

\title{ When Is a Local Homeomorphism a Semicovering Map? }
\author[fumadep]{ Majid Kowkabi }
\ead{m.kowkabi@stu.um.ac.ir}
\author[fumadep]{Behrooz Mashayekhy \corref{cor1}}
\ead{bmashf@um.ac.ir}
\author[fumadep]{Hamid Torabi}
\ead{h.torabi@ferdowsi.um.ac.ir}

\address[fumadep]{Department of Pure Mathematics, Center of Excellence in Analysis on Algebraic Structures, Ferdowsi University of Mashhad,
P.O.Box 1159-91775, Mashhad, Iran }
\cortext[cor1]{Corresponding author}


\begin{abstract}
In this paper, by reviewing the concept of semicovering maps, we present some conditions under which a local homeomorphism becomes a semicovering map. We also obtain some conditions under which a local homeomorphism is a covering map.

\end{abstract}


\end{frontmatter}

\section{\bf Introduction}
It is well-known that every covering map is a local homeomorphism. The converse seems an interesting question that when a local homeomorphism is a covering map (see \cite{Wenyan}, \cite{Lelek}, \cite{Jungck}.) Recently, Brazas \cite[Definition 3.1]{B1} generalized the concept of covering map by the phrase {\it {``A semicovering map is a local homeomorphism with continuous lifting of paths and homotopies''}}. Note that a map $p: Y \to X $ has {\it{continuous lifting of paths}} if $ \rho_p: (\rho Y )_y \to (\rho X)_{p(y)}$ defined by $\rho_p(\alpha) = p \circ \alpha $ is a homeomorphism for all $y \in Y,$ where $(\rho Y)_y = \{\alpha: [0, 1] \to Y | \alpha(0) = y\}$. Also, a map $p: Y \to X $ {\it{has continuous lifting of homotopies}} if $\Phi_p: (\Phi Y )_y \to (\Phi X)_{p(y)}$
defined by $\Phi_p(\phi) = p \circ \phi$ is a homeomorphism for all $y \in Y$, where elements of $(\Phi Y)_y$ are endpoint preserving homotopies of paths starting at $y$.
It is easy to see that any covering map is a semicovering.

The quasitopological fundamental group $\pi_1^{qtop}(X,x)$ is the quotient space of the loop space
$\Omega (X,x) $ equipped with the compact-open topology with respect to the function $\Omega (X,x) \to \pi_1(X,x)$
identifying path components (see \cite{B}).
Biss \cite[Theorem 5.5]{B} showed that for a connected, locally path connected space $X$, there
is a one to one correspondence between its equivalent classes of connected covering spaces and the conjugacy classes of open subgroups of
its fundamental group $\pi_1^{qtop}(X,x)$. There is a misstep in the proof of the above theorem.
Torabi et al. \cite{t} pointed out the above misstep and gave the true classification of connected covering spaces of $X$ according to subgroups of the fundamental group  $\pi_1^{qtop}(X,x)$ with open core. Using this classification, it can be concluded that for a locally path connected space $X$ a semicovering map $p:\tilde{X} \To X$ is a covering map if and only if the core of $p_*(\pi_1(\tilde{X},\tilde{x}_0))$ in $ \pi_1(X,x_0)$ is an open subgroup of $\pi_1^{qtop}(X,x_0)$. By using this fact, we give some conditions under which a semicovering map becomes a covering map which extend some results of \cite{Wenyan}.

In Section 2, among reviewing the concept of local homeomorphism, path lifting property and unique path lifting property, we mention a result on uniqueness of lifting for local homeomorphism and a simplified definition \cite[Corollary 2.1]{Klevdal} for a semicovering map. Note that there is a misstep in the simplification of semicovering  which we give another proof to remedy this defect.
In Section 3, we intend to find some conditions under which a local homeomorphism is a semicovering map. Among other things, we prove that if $p:\tx \to X $ is a local homeomorphism, $\tx$ is Hausdorff and sequential compact, then $p$ is a semicovering map. Also, a closed local homeomorphism from a Hausdorff space is a semicovering map.
Moreover, a proper local homeomorphism from a Hausdorff space onto a Hausdorff space is a semicovering..

Finally in Section 4, we generalize some results of \cite{Wenyan}. In fact, by openness of the core of $p_*(\pi_1(\tilde{X},\tilde{x}_0))$ in $ \pi_1(X,x_0)$, for a local homeomorphism $p:\tilde{X} \To X$, we obtain some conditions under which a semicovering map is a covering map. More precisely, we prove that every finite sheeted semicovering map on a locally path connected space is a covering map. Also, a proper local homeomorphism from a Hausdorff space onto a locally path connected Hausdorff space is a covering map.

\section{Notations and Preliminaries}

In this paper, all maps $f : X \to Y$ between topological spaces
$X$ and $Y$ are continuous functions. We recall that a continuous map
$p:\tilde{X} \To X$,
is called a {\it local homeomorphism} if for every point
$\tilde{x} \in \tilde{X}$
there exists an open neighbourhood $\tilde{W} $ of $\tilde{x}$ such that $p(\tilde{W}) \subset X$ is open and the restriction map  $p|_{ \tilde{W}}: \tilde{W} \To p( \tilde{W})$ is a homeomorphism. In this paper, we denote a local homeomorphism $p:\tilde{X} \To X$ by $( \tx,p)$ and assume that $\tx$ is path connected and $p$ is surjective.
\begin{defn}
Assume that $X$ and $\tx$ are topological spaces and $p:\tilde{X} \To X$ is a continuous map.
Let $f:(Y, y_0) \to (X, x_0)$ be a continuous map and $\tilde{x}_0\in p^{-1}(x_0)$. If there exists a continuous map $\tilde{f}:(Y, y_0) \to (\tx, \tilde{x}_0)$ such that $p \circ \tilde{f} =f $, then  $\tilde{f}$ is called a {\it lifting} of $f$.

The map $p$ has {\it path lifting property} if for every path $f$ in $X$, there exists a lifting $\tilde{f}:(I, 0) \to (\tx, \tilde{x}_0)$ of $f $.
Also, the map $p$ has {\it unique path lifting property} if for every path $f$ in $X$, there is at most one lifting $\tilde{f}:(I, 0) \to (\tx, \tilde{x}_0)$ of $f $ (see \cite{r}.)
\end{defn}

The following lemma is stated in \cite[Lemma 5.5]{F} for $Y =I$. One can state it for an arbitrary map $f: X \to Y$ for a connected space $Y$.
\begin{lem}\label{unique}
Let $(\tx, p)$ be a local homeomorphism on $X$, $Y$ be a connected space, $\tx$ be Hausdorff and $f:(Y, y_0) \to (X, x_0)$ be a continuous map. Given $\tilde{x}\in p^{-1}(x_0)$ there is at most one lifting $\tilde{f}:(Y, y_0) \to (\tx, \tilde{x}_0)$ of $f.$
\end{lem}

The following theorem can be concluded from \cite[Definition 7, Lemma 2.1, Proposition 2.2]{Klevdal}.

\begin{thm} \label{semicovering map}(\cite[Corollary 2.1]{Klevdal}).
A map $p:\tilde{X} \To X$ is a semicovering map if and only if it is a local homeomorphism with unique path lifting and path lifting properties.
\end{thm}

Note that during reviewing process of the previous version of this paper, the referee pointed out that there exists a misstep in the proof of \cite[Lemma 2.1]{Klevdal}. More precisely, in the proof of \cite[Lemma 2.1]{Klevdal}, it
is not guaranteed that $h_t(K^j_n) \cap U)\neq \phi$, i.e., $h_t|_{K^j_n}$ might be a lift of $\gamma$ that is different from $(p|_U)^{-1}\circ \gamma$. After some attempts to find a proof for the misstep, we found out that the method in the proof of \cite[Lemma 2.1]{Klevdal} dose not work. Now, we give a proof for \cite[Lemma 2.1]{Klevdal} with a different method
as follows.

\begin{lem} \label{Homotopy}
(Local Homeomorphism Homotopy Theorem).
Let $p:\tilde{X} \To X$ be a local homeomorphism with unique path lifting and path lifting properties. Consider the diagram of continuous maps

\[
\xymatrix{
 I \ar[d]^{j} \ar[r]^{\tilde f} & (\tx, \tilde{x}_0) \ar[d]_p \\
 I\times I \ar[r]^F \ar@{-->}[ur]^{\tilde F} & (X, x_0),}
\]
where $j(t)= (t, 0)$ for all $t \in I$. Then there exists a unique continuous map $ \tilde{F}:I\times I \to \tx$ making the diagram commute.
\end{lem}

\begin{proof}
Put $ \tilde{F}(t,0)= \tilde{f}(t)$,  for all $ t \in I $, and  $W_{\tilde{f}(t)}$ an open neighborhood of $\tilde{f}(t)$ in $ \tx $ such that $ p |_{W_ {\tilde{f}(t)}}:
W_{ \tilde{f}(t)} \to p(W_ {\tilde{f}(t)})$ is a homeomorphism. Then $\{\tilde{f}^{-1}(W_{\tilde{f}(t)})|t \in I\} $ is an open cover for $I$. Since $I$ is compact, there
exists $n \in \NN$ such that for every $0 \leq i \leq n-1$ the interval $[\frac{i}{n}, \frac{i+1}{n}]$ is contained in $\tilde{f}^{-1}(W_{\tilde{f}(t_i)})$
for some $t_i \in I$. For every $0 \leq i \leq n-1$, $ F^{-1}(p(W_{{\tilde{f}}(t_i)}))$ is open in $I \times I$ which contains $(\frac{i}{n},0)$. Hence
there exists $s_i \in I$ such that $[\frac{i}{n}, \frac{i+1}{n}] \times [0, s_i]$ is contained in $ F^{-1}(p(W_{{\tilde{f}}(t_i)}))$ and so $F([\frac{i}{n},
 \frac{i+1}{n}] \times [0, s_i]) \subseteq p(W_{{\tilde{f}}(t_i)}) $. Since $ p |_{W_ {\tilde{f}(t_i)}}:W_{ \tilde{f}(t_i)} \to p(W_ {\tilde{f}(t_i)})$ is
a homeomorphism, we can define $\tilde{F}$ on $k_i = [\frac{i}{n}, \frac{i+1}{n}] \times [0, s_i] $ by $p^{-1}|_{W_{\tilde{f}(t_i)}} \circ F|_{k_i}$. Let
 $s = \min \{s_i|0 \leq i \leq n-1\}$, then by gluing lemma, we can define $\tilde{F}$ on $I\times[0, s]$ since
$ \{ \frac{i+1}{n}\} \times [0, s] \in k_i \cap k_{i+1}.$  Put $A = \{r \in I| \ \mathrm {there} \ \mathrm {exists} \ \tilde{F_r}:I \times [0, r] \to
\tx \ \mathrm {such \ that} \ F(x, y) = p \circ \tilde{F_r}(x, y)  \mathrm {\ for \ every \ } (x,y) \in I \times [0, r]\ \mathrm {and}\ \tilde{F}(t, 0) = \tilde{f}(t)\}.$
 Note that $A$ is a nonempty since $ 0 \in A$. We show that there exists $M \in I$ such that $M = \max{A}$. For this, consider an increasing sequence
  $\{a_n\}_{n \in \NN}$ in $A$ such that $a_n \to a$. We show that $a \in A$. Let $s < a$, then there exists $n_s \in \NN$ such that $ s \leq a_{n_s}$.
We define $ H: I \times [0, a] \to \tx$ by $$  H(t, s) =  \begin{cases}  \tilde{F}_{a_{n_s}}(t, s) & s < a \\ \lambda_a(t) & s = a,\end{cases}
$$ where $\lambda_a(t)$ is the lifting of
the path $F (I\times \{a\})$ starting at $\gamma(a)$ and $\gamma$ is the lifting of the path $F(0, t)$ for $t \in I$ starting at $\tilde{f}(0)$. Note that the existence of
 $\lambda_a(t)$ and $\gamma$ are due to path lifting property of $p$. The map $H|_{I \times [0,a)}$ is well defined and continuous since if $r_1, r_2 \in A$ and
 $r_1 < r_2$, then $\tilde{F}_{r_2}|_{I\times [0, r_1)} = \tilde{F}_{r_1}$ and $p$ is a local homeomorphism with unique path lifting property. Therefore  $H|_{\{0\} \times [0, a)}$ is a lifting of $F|_{\{0\} \times [0, a)}$ starting at $\tilde{f}(0)$ which implies that $H(0, t)= \gamma (t)$ for every $0\leq t < a$.
Put $B = \{t \in I| H|_{\{I \times [0,a)\}\cup \{[0, t] \times \{a\}\}} \mathrm {is \ continuous} \}$. We show that, there exists $0 < \epsilon < 1$ such that $ H|_{I \times [0,a) \cup \{[0, \epsilon)\times \{a\}\}}$ is continuous.
For this, consider $U = U_{\gamma(a)}$ an open neighborhood of $\gamma(a)$ such that $p|_{U_{\gamma(a)}} : U_{\gamma(a)} \to p(U_{\gamma(a)})$ is
a homeomorphism. There exists $0 < \epsilon < 1$ such that $F ([0,\epsilon] \times [a - \epsilon, a]) \subseteq p(U_{\gamma(a)})$ so we have a lifting of $F|_{[0,\epsilon] \times [a - \epsilon, a]}$ in $U_{\gamma(a)}$ by $p^{-1}|_U \circ F|_{[0,\epsilon] \times [a - \epsilon, a]}$. Note that $p^{-1}|_U  \circ F|_{[0, \epsilon] \times [a-\epsilon, a]}(0, t)$ and $  \gamma(t)$ are two liftings of $F|_{\{0\} \times [a - \epsilon, a]}(0, t)$ such that $p^{-1}|_U  \circ F(0, a)= \lambda _a(0) = \gamma (a)$. By unique path lifting property we have $p^{-1}|_U  \circ F|_{[0, \epsilon] \times [a-\epsilon, a]}(0, t) =  \gamma(t)$ for $t \in [a-\epsilon, a] $ so $p^{-1}|_U  \circ F|_{[0, \epsilon] \times [a-\epsilon, a]}(0, a- \frac{\epsilon}{2})= \gamma(a-\frac{\epsilon}{2})$. Since $H (0, a- \frac{\epsilon}{2}) = \gamma(a-\frac{\epsilon}{2})$, by unique path lifting property we have $p^{-1}|_U  \circ F|_{[0, \epsilon] \times [a-\epsilon, a)} = H|_{[0, \epsilon] \times [a-\epsilon, a)}$. Note that $p^{-1}|_U  \circ F|_{[0, \epsilon] \times [a-\epsilon, a]}(t, a)$ and $  \lambda_a(t)$ are two liftings of $F|_{[0, \epsilon] \times\{a\} }(t, a)$ such that $p^{-1}|_U  \circ F (0, a) = \gamma(a)= \lambda_a(0) $. By unique path lifting property, we have $p^{-1} \circ F (t, a) = \lambda_a(t) = H(t, a)$ for $t \in [0, \epsilon]$. Hence $ H|_{I \times [0,a) \cup \{[0, \epsilon)\times \{a\}\}}$ is continuous which implies that $B$ is nonempty. We show that $B$ has a maximum element and $\max{B} =1$. For this, consider an increasing sequence
  $\{b_n\}_{n \in \NN}$ in $B $ such that $b_n \to b$. We know that $H$ is continuous on $\{I \times [0,a)\} \cup \{[0, b) \times \{a\}\}$. By a similar argument for the continuity of $ H|_{I \times [0,a) \cup \{[0, \epsilon)\times \{a\}\}}$, we can prove that $H$ is continuous on $\{I \times [0,a)\} \cup \{[0, b] \times \{a\}\}$ and  $\max B =1$. Thus $B = I$.
Therefore $a \in A$, which implies that $A$ has a maximum.
Finally, by a similar idea for constructing $\tilde{F}$ on $ I \times [0, s]$ we can show that $M =1$. Hence we have a lifting for $F$ by $p$ making the above diagram commute.
Uniqueness of $\tilde{F}$ is obtained by Lemma \ref{unique} since $ I \times I$ is connected.
\end {proof}

Note that there exists a local homeomorphism without unique path lifting and path lifting properties and so it is not a semicovering map .
\begin{ex}\label{ex}
Let $\tx = ( [0, 1] \times \{0 \} ) \bigcup (\{1/2\} \times [0, 1/2))$ with coherent topology with respect to $\{ [0, 1/2] \times \{0 \}, \ (1/2, 1] \times \{0 \}, \ \{1/2\} \times (0, 1/2)  \}$ and let $X = [0, 1]$. Define $p:\tx \to X$ by $$  p(s, t) =  \begin{cases}  s & t =0 \\ s+1/2 & s = 1/2 \end{cases}. $$ It is routine to check that $p$ is an onto local homeomorphism which dose not have unique path lifting and path lifting properties.
\end{ex}

\section{\bf When Is a Local Homeomorphism a Semicovering map?}

In this section, we obtained some conditions under which a local homeomorphism is a semicovering map.
First, we intend to show that if $p:\tx \to X $ is a local homeomorphism, $\tx$ is Hausdorff and sequential compact, then $p$ is a semicovering map. In order to do this,
we are going to study a local homeomorphism with a path which has no lifting.
\begin{lem}\label{exist}
Let $p:\tx \to X $ be a local homeomorphism, $f$ be an arbitrary path in $X$ and $ \tilde{x}_0 \in p^{-1}(f(0))$ such that there is no lifting of $f$
starting at $\tilde{x}_0$. If $ A_f = \{ t \in I \ | f|_{[0,t]} \ {\mathrm has} \mathrm{ \ a \ lifting} \  \hat{f}_t \ \mathrm{on} \  [0,t] \ \mathrm{with} \ \hat{f}_t(0)  = \tilde{x}_0\}$, then $A_f$ is open and connected. Moreover, there exists $ \alpha \in I $ such that $A_f = [0, \alpha)$.
\end{lem}

\begin{proof}
Let $\beta$ be an arbitrary element of $A_f$. Since $ p$ is a local homeomorphism, there exists an open neighborhood $W$ at $\hat{f}_\beta(\beta)$ such that $p|_{W}: W \to p(W)$ is a homeomorphism. Since $\hat{f}_\beta(\beta) \in W$, there exists an $\epsilon \in I$ such that $f[\beta, \beta +\epsilon] $ is a subset of $p(W)$. We can define a map $\hat{f}_{\beta + \epsilon}$ as follows:
$$ \hat{f}_{\beta + \epsilon}(t) =  \begin{cases}  \hat{f}_\beta(t) & t \in [0, \beta] \\ p|_{W}^{-1}(f(t)) & t \in [\beta, \beta +\epsilon] \end{cases}.$$
Hence $(0, \beta + \epsilon)$ is a subset of $A_f$ and so $A_f$ is open.

Suppose $t,s \in A$. Without loss of generality we can suppose that $ t \geq s $. By the definition of $A_f$, there exists $ \hat{f}_t$ and so $[0, t]$ is a subset of $A_f$. Also, every point between $s$ and $t$ belongs to $A_f$ hence $A_f$ is connected.
Since $A_f$ is open connected and $0 \in A_f$, there exists $\alpha \in I $ such that $A_f =[0, \alpha)$.
\end {proof}

Now, we prove the existence and uniqueness of a concept of a defective lifting.
\begin{lem}\label{cts}
let $p:\tx \to X $ be a local homeomorphism with unique path lifting property, $f$ be an arbitrary path in $X$ and $ \tilde{x}_0 \in p^{-1}(f(0))$, such that there is no lifting of $f$ starting at $\tilde{x}_0$. Then, using the notation of the previous lemma, there exists a unique continuous map $\tilde{f}_{\alpha}: A_f = [0,\alpha) \to \tx$ such that $p\circ\tilde{f}_{\alpha} = f|_{[0,\alpha)}$.
\end{lem}
\begin{proof}
First, we defined $\tilde{f}_{\alpha}: A_f = [0,\alpha) \to \tx$ by $\tilde{f}_{\alpha}(s) = \hat{f}_s(s)$. The map $\tilde{f}_{\alpha}$ is well define since if $s_1 = s_2 $, then by unique path lifting property of $p$ we have  $\hat{f}_{s_1} = \hat{f}_{s_2}$ and so $\hat{f}_{s_1}(s_1) = \hat{f}_{s_2}(s_2)$ hence $\tilde{f}_{\alpha}(s_1) = \tilde{f}_{\alpha}(s_2)$. The map $\tilde{f}_{\alpha}$ is continuous since for any element $s$ of $A_f$, $\hat{f}_{\frac{\alpha+s}{2}}$ is continuous at $s$ and $\hat{f}_{\frac{\alpha+s}{2}} = \hat{f}_{s}$ on $[0, s]$. Thus there exists $\epsilon > 0$ such that $\tilde{f_{\alpha}}|_{(s-\epsilon,  s+\epsilon)}= \hat{f}_{\frac{\alpha+s}{2}}|_{(s-\epsilon,  s+\epsilon)} $. Hence $\tilde{f}_{\alpha}$ is continuous at $s$. For uniqueness, if there exists $\hat{f}_{\alpha}:[0,\alpha) \to \tx$ such that $p\circ\hat{f}_{\alpha} = f|_{[0,\alpha)}$, then by unique path lifting property of $\tx $ we must have $\tilde{f}_{\alpha} = \hat{f}_{\alpha}$.
\end {proof}
\begin{defn}
By Lemmas \ref{exist} and \ref{cts}, we called $\tilde{f}_{\alpha}$ the {\it incomplete lifting} of $f$ by $p$ starting at $\tilde{x}_0$.
\end{defn}
\begin{thm}\label{sequential}
If $\tilde{X}$ is Hausdorff and sequential compact and $p:\tilde{X} \To X$ is a local homeomorphism, then $p$ has the path lifting property.
\end{thm}
\begin{proof}
Let $f: I \to X$ be a path which has no lifting starting at $\tilde{x}_0 \in p^{-1}(f(0))$. Using the notion of Lemma \ref{exist}, let $\tilde{f}: A_f = [0, \alpha) \to \tx$ be the incomplete lifting of $f$ at $\tilde{x}_0$. Suppose $ \{ t_n\}_{0}^ \infty$ is a sequence in $A_f$ which tends to $t_0$ and $t_n \leq t_0$. Since $ \tx$ is a sequential compact, there exists a convergent subsequence of ${\tilde{f}(t_n)}$, $ \{\tilde{f}(t_{n_k})\}_{0}^ \infty$ say, such that $\tilde{f}(t_{n_k})$ tends to $l$.
We define
$$ g(t) =  \begin{cases}  \tilde{f}(t) & 0\leq t <t_0 \\ l = \lim _{k \rightarrow \infty} \tilde{f}(t_{n_k}) & t =t_0.\end{cases} $$
We have $p (l) = p ( \lim_{k \rightarrow \infty} \tilde{f}(t_{n_k}))= \lim _{k \rightarrow \infty} p( \tilde{f}(t_{n_k}))= \lim _{k \rightarrow \infty} f(t_{n_k})=f(t_0)$ and so $p \circ g = f$. We show that $g$ is continuous on $[0, t_0]$, for this we show that $g$ is continuous at $t_0$. Since $p$ is a local homeomorphism, there exists a neighborhood $W$ at $l$ such that $p|_W:W \to p(W)$ is a homeomorphism. Hence there is $a \in I$ such that $f([a, t_0]) \subseteq p(W)$. Let $V$ be a neighborhood at $l$ and $W' = V \cap W$, then $ p(W') \subseteq  p(W)$ is an open set. Put $U = f^{-1}(p(W')) \cap (a, t_0] $ which is open in $[0, t_0]$ at $t_0$. It is enough to show that $ g(U) \subseteq W'$. Since $f(U) \subseteq p(W')$ and $p$ is a homeomorphism on $W'$, $(p|_W)^{-1}(f(U))\subseteq (p|_W)^{-1}(p(W'))$ and so $(p|_W)^{-1}\circ f = \tilde{f}$ on $[a, t_0)$ since $p(l) = f(t_0)$. Hence $(p|_W)^{-1} \circ f = g$ on $[a, t_0]$. Thus $ g(U) \subseteq g (f^{-1}(p(W')) ) = (p|_W)^{-1} \circ f (f^{-1}(p(W')) ) \subseteq  (p|_W)^{-1} \circ p(W') \subseteq W' \subseteq V$, so $g$ is continuous. Hence $t_0 \in A_f$, which is a contradiction.
\end {proof}
\begin{cor}\label{sequential semi}
If $\tilde{X}$ is Hausdorff and sequential compact and $p:\tilde{X} \To X$ is a local homeomorphism, then $p$ is a semicovering map.
\end{cor}
\begin{proof}
Since $\tilde{X}$ is Hausdorff and sequential compact, by Theorem \ref{sequential}, the map $p$ has path lifting property and by Lemma \ref{unique} $p$ has unique path lifting property. Hence Theorem \ref{semicovering map} implies that $p$ is a semicovering map.
\end {proof}

Chen and Wang \cite[Theorem 1]{Wenyan} showed that a closed local homeomorphism $p$, from a Hausdorff space $\tx$ onto a connected space $X$ , is a covering map, when there exists at least one point $x_0 \in X$ such that $|p^{-1}(x_0)| = k$, for some finite number $k$. In the following theorem, we extend this result for semicovering map without finiteness condition on any fiber.
\begin{thm}\label{closed}
Let $p$ be a closed local homeomorphism from a Hausdorff space $\tx$ onto a space $X$. Then $p$ is a semicovering map.
\end{thm}
\begin{proof}
Using Theorem \ref{semicovering map}, it is enough to show that $p$ has unique path lifting and path lifting properties. By Lemma \ref{unique}, $p$ has the unique path lifting property. To prove the path lifting property for $p$, suppose there exists a path $f$ in $X$ such that it has no lifting starting at $\tilde{x}_0 \in p^{-1}(f(0))$. Let $g: A_f = [0, \alpha) \to \tx$ be the incomplete lifting of $f$ at $\tilde{x}_0$. Suppose $ \{ t_n\}_{0}^ \infty$ is a sequence which tends to $\alpha$.
Put $B = \{t_n | n \in \NN\}$, then $\overline {g(B)}$ is closed in $\tx$ and so $p(\overline {g(B)})$ is closed in $X$ since $p$ is a closed map. Since $ \overline {f(B)} \subseteq p(\overline {g(B)})$,  $f(\alpha) \in p(\overline {g(B)})$ and so there exists $\beta \in \overline {g(B)}$ such that $p(\beta)= f(\alpha)$. Since $p$ is a local homeomorphism, there exists a neighborhood $W_{\beta}$ at $\beta$ such that $p|_{W_{\beta}} : W_{\beta} \to p(W_{\beta})$ is a homeomorphism. Since $\beta \in \overline {g(B)}$, there exists an $n_k \in \NN$ such that $g(t_{n_r}) \in W_{\beta}$ for every $n_r \geq n_k$. Since $f(\alpha) \in p(W_{\beta})$, there exists $k_1 \in \NN$ such that $f([t_{n_{k_1}}, \alpha]) \subseteq p( W_{\beta})$. Put $h = ((p|_{W_{\beta}})^{-1}\circ f)|_{[t_{n_{k_1}}, \alpha]}$, then $p \circ g = f = p \circ h$ on $[t_{n_{k_1}}, \alpha)$. Hence $p \circ h = p \circ g$ on $[t_{n_{k_1}}, \alpha)$. Since $p|_{W_{\beta}}$ is a homeomorphism, $g (t_{n_{k_1}})=h (t_{n_{k_1}})$, thus $g = h $ on $[t_{n_{k_1}}, \alpha)$. Therefore the map $\bar{g}: [0, \alpha] \to \tx$ defined by
$$ \bar{g}(t) =  \begin{cases}  g(t) & 0\leq t <\alpha \\ h(\alpha) = \beta & t =\alpha \end{cases} $$
is continuous and $p \circ \bar{g} = f $ on $[0, \alpha]$. Hence $\alpha \in A_f $, which is a contradiction.
\end{proof}

A map $f : X \to Y $ is called {\bf proper} if and only if $f^{-1}(H)$ is compact for any compact subset $ H$ of $ Y$.
\begin{thm}\label{proper}
If $p$ is a proper local homeomorphism from a Hausdorff space $\tx$ onto a Hausdorff space $X$, then $p$ is a semicovering map.
\end{thm}
\begin{proof}
The map $p$ is open since every local homeomorphism is an open map. Also, every open proper map is closed (see \cite[Fact 2.3]{Timm}). Hence using Theorem \ref{closed}, $p$ is a semicovering map.
\end{proof}

\section{\bf When Is a Semicovering Map a Covering Map?}

If $p:\tilde{X}\to X  $ is a semicovering, then $\pi_1(X, x_0)$ acts on $Y=p^{-1}(x_0)$ by $\alpha \tilde{x}_0 = \tilde{\alpha}(1)$, where $\tilde{x}_0 \in Y $ and $\tilde{\alpha}$ is the lifting of $\alpha$ starting at $\tilde{x}_0$ (see \cite{B1}). Therefore, we can conclude that the stabilizer of $\tilde{x}_0$, $\pi_1(X, x_0)_{\tilde{x}_0}$, is equal to $p_*(\pi_1(\tilde{X}, \tilde{x}_0))$ for all $\tilde{x}_0 \in Y$ and so $|Y|=[\pi_1(X, x_0):p_*(\pi_1(\tilde{X}, \tilde{x}_0))]$. Thus, if $x_0, x_1 \in X $, $Y_0=p^{-1}(x_0)$, $Y_1=p^{-1}(x_1)$ and $\tx$ is a path connected space, then $|Y_0|=|Y_1|$. Hence one can define the concept of sheet for a semicovering map similar to covering maps.

The following example shows that there exists a local homeomorphism with finite fibers which is not a semicovering map.
\begin{ex}\label{ex1}
Let $\tx = (0, 2)$  and let $X = S^1$. Define $p:\tx \to X$ by $  p(t) = e^{2 \pi it} $. It is routine to check that $p$ is an onto local homeomorphism which their fibers are finite but $p$ is not a semicover since
\[
|p^{-1}((0, 1))| = 2 \neq 1 = |p^{-1}((1, 0))|.
\]
\end{ex}

Note that the core of a subgroup $H$ of $G$, denoted by $H_G$, is defined to be the join of all the normal subgroups of $G$ that are contained in $H$. It is easy to see that $H_G =\bigcap_{g \in G}g^{-1}Hg.$
\begin{thm} \label{pp}
Suppose $p:(\tilde{X}, \tilde{x}_0) \to (X, x_0) $ is a semicovering map and $X$ is locally path connected such that $[\pi_1(X,x_0): p_*(\pi_1(\tilde{X},\tilde{x}_0))]$ is finite, then $p$ is a covering map.
\end{thm}
\begin{proof}
Since $p$ is a semicovering map, $p_*(\pi_1(\tilde{X}, \tilde{x}_0))$ is an open subgroup of $\pi_{1}^{qtop}(X, x_0)$ (see \cite[Corollary 3.4]{B2}.) Put $[\pi_1(X, x_0): p_*(\pi_1(\tilde{X}, \tilde{x}_0))]= m$. It is well-known that $p_*(\pi_1(\tilde{X}, \tilde{x}_0))$ is a subgroup of the normalizer of $p_*(\pi_1(\tilde{X}, \tilde{x}_0))$ in $\pi_1(X, x_0)$, $N_{\pi_1(X, x_0)}(p_*(\pi_1(\tilde{X}, \tilde{x}_0)))$ and so $[\pi_1(\tilde{X}, \tilde{x}_0):N_{\pi_1(X, x_0)}(p_*(\pi_1(\tilde{X}, \tilde{x}_0)))]$ is finite. Therefore the core of $p_*(\pi_1(\tilde{X}, \tilde{x}_0))$ in $ \pi_1(X, x_0)$, $p_*(\pi_1(\tilde{X}, \tilde{x}_0))_{\pi_1(X, x_0)}$  is the intersection of a finitely many conjugations of $ p_*(\pi_1(\tilde{X}, \tilde{x}_0))$. Since any conjugation of $p_*(\pi_1(\tilde{X}, \tilde{x}_0))$ is open, $p_*(\pi_1(\tilde{X}, \tilde{x}_0))_{\pi_1(X, x_0)}$ is also open in $\pi_{1}^{qtop}(X, x_0)$. By the classification of connected covering spaces of $X$ (see \cite{t}), there is a covering map $q : (\tilde{Y}, \tilde{y_0}) \to (X, x_0)$ such that $q_*(\pi_1(\tilde{Y}, \tilde{y_0}))= p_*(\pi_1(\tilde{X}, \tilde{x}_0))$. Hence $p, q$ are equivalent as semicoverings. Therefore $p$ is a covering map.
\end{proof}

The following result is an immediate consequence of the above theorem.
\begin{cor} \label{finite sheeted semicovering}
Every finite sheeted semicovering map on a locally path connected space $X$ is a covering map.
\end{cor}

Note that Brazas presented an infinite sheeted semicovering map which is not a covering map (see \cite[Example 3.8]{B1}).
Moreover, Theorem 1 in \cite{Wenyan} is an immediate consequence of our results Theorem \ref{closed} and Corollary \ref{finite sheeted semicovering}.

It is well-known that a proper local homeomorphism from a Hausdorff space to a locally compact, Hausdorff space is covering map. Chen and Wang  \cite[Corollary 2]{Wenyan} showed that a proper local homeomorphism from a Hausdorff, first countable space onto a Hausdorff, connected space is a covering map. We extend this result without first countability as follows.
\begin{thm}
If $p$ is a proper local homeomorphism from a Hausdorff space $\tx$ onto a locally path connected, Hausdorff space $X$, then $p$ is a finite sheeted covering map.
\end{thm}
\begin{proof}
 The local homeomorphism $p$ is a semicovering map by Theorem \ref{proper}. Since singletons in $X$ are compact and $p$ is proper, every fiber of $p$ is compact in $\tx$. Since fibers of a local homeomorphism are discrete, fibers of $p$ are compact and so they are finite. Therefore $p$ is a finite sheeted semicovering map. Hence by Corollary \ref{finite sheeted semicovering}, $p$ is a covering map.
\end{proof}
\begin{cor}
If $\tilde{X}$ is Hausdorff and sequential compact, $X$ is locally path connected and $p:\tilde{X} \To X$ is a local homeomorphism  with a finite fiber, then $p$ is a covering map.
\end{cor}
\begin{proof}
By Corollary \ref{sequential semi}, $p$ is a semicovering map. Hence $p$ is a finite sheeted semicovering map which is a covering map by Corollary \ref{finite sheeted semicovering}.
\end {proof}

\bibliographystyle{plain}

\end{document}